\documentclass[12pt]{article}

\usepackage[utf8]{inputenc}
\usepackage[T1]{fontenc}

\usepackage{amsmath,mleftright,mathtools,
            amssymb,amsthm,nicefrac,xcolor,
            enumerate,color,comment,enumitem,
            mathrsfs,bbm,stmaryrd,geometry
            }
            
\mleftright   

\usepackage[sort,nocompress]{cite}

\usepackage[colorinlistoftodos]{todonotes}
\tikzset{notestyleraw/.append style={align=justify}}
                        
\usepackage[english]{datetime2}
            
\usepackage[colorlinks=true]{hyperref}
            
\usepackage[sort,capitalize]{cleveref}

\crefformat{equation}{(#2#1#3)}
\crefname{enumi}{item}{items}
\crefname{subsection}{Subsection}{Subsections}

\newtheorem{theorem}{Theorem}[section]
\newtheorem{definition}[theorem]{Definition}

\newtheorem{corollary}[theorem]{Corollary}
\newtheorem{lemma}[theorem]{Lemma}
\newtheorem{setting}[theorem]{Setting}

\numberwithin{equation}{section}

\usepackage[makeroom]{cancel}

\newcounter{todocounter}

\newcommand{\N}{\ensuremath{\mathbb{N}}}

\newcommand{\R}{\ensuremath{\mathbb{R}}}
\newcommand{\C}{\ensuremath{\mathbb{C}}}

\newcommand{\logNorm}{\mu_X}

\newcommand{\semigroup}{\mathbb{T}}
\newcommand{\real}{\mathfrak{R}}
\newcommand{\id}{\mathbb{I}_X}

\newcommand{\domain}{D}
\newcommand{\resolve}{\rho_X}

\newcommand{\wright}{\Psi}
\newcommand{\mittag}{E}
\newcommand{\fractional}{S}

\newcommand{\with}{\curvearrowleft}

\DeclarePairedDelimiter{\pr}{(}{)}
\DeclarePairedDelimiter{\br}{[}{]}
\DeclarePairedDelimiter{\cu}{\{}{\}}
\DeclarePairedDelimiter{\abs}{\lvert}{\rvert}
\DeclarePairedDelimiter{\norm}{\lVert}{\rVert_X}

\DeclarePairedDelimiter{\inner}{[}{]_X}

\allowdisplaybreaks

\begin{document}

\title{Intrinsic properties of strongly continuous fractional semigroups in normed vector spaces}

\author{
Tiffany Frug{\'e} Jones$^1$,
Joshua Lee Padgett$^{2,3}$, and Qin Sheng$^{3,4}$
\bigskip
\\
\small{$^1$ Department of Mathematics, The University of Arizona,}
\vspace{-0.1cm}\\
\small{Tucson, Arizona, USA, e-mail: \texttt{tnjones@math.arizona.edu}}
\smallskip
\\
\small{$^2$ Department of Mathematical Sciences, University of Arkansas,}
\vspace{-0.1cm}\\
\small{Fayetteville, Arkansas, USA, e-mail: \texttt{padgett@uark.edu}}
\smallskip
\\
\small{$^3$ Center for Astrophysics, Space Physics, and Engineering Research,}
\vspace{-0.1cm}\\
\small{Baylor University, Waco, Texas, USA}
\smallskip
\\
\small{$^4$ Department of Mathematics, Baylor University,}
\vspace{-0.1cm}\\
\small{Waco, Texas, USA, e-mail: \texttt{qin\_sheng@baylor.edu}}
\smallskip
}

\date{\today}

\maketitle

\begin{abstract}
Norm estimates for strongly continuous semigroups have been successfully studied in numerous settings, but at the moment there are no corresponding studies in the case of solution operators of singular integral equations. Such equations have recently garnered a large amount of interest due to their potential to model numerous physically relevant phenomena with increased accuracy by incorporating so-called non-local effects. In this article, we provide the first step in the direction of providing such estimates for a particular class of operators which serve as solutions to certain integral equations. The provided results hold in arbitrary normed vector spaces and include the classical results for strongly continuous semigroups as a special case.
\end{abstract}


\tableofcontents


\section{Introduction}
\label{sec:intro}

The use of operator theory has been proven to be extremely effective in numerical analysis. It allows 
for the development of robust tools for applications far outside of the realms of traditional
computational mathematics. While these tools have been important and are well understood in 
cases of classical abstract Cauchy problems---which may be used to represent continuous and discrete 
problems, alike---there is currently little known about their effectiveness when abstract singular integral 
problems are presented. In particular, there is no accountable analysis tools available 
if the generating operators considered are not sectorial.

Researchers in numerical analysis have recently benefited from a surge of activities 
which overlaps with pure mathematical analysis (see, for instance, 
\cite{iserles2000lie,%
munthe2018lie,curry2017post,orthogonal2006iserles,%
padgett2020analysis,padgett2019convergence,%
arnold2010finite}
and publications cited therein). Inspired by this fact and the recent works of Littlejohn and Wellman concerning the investigation of self-adjoint operators in extended 
Hilbert spaces (cf., e.g., \cite{lance2019,lance2016,lance2002}), in this article we consider so-called \textit{strongly continuous fractional semigroups} in arbitrary normed vector spaces.
The work of Littlejohn and Wellman has added to the existing notion that understanding operators in a more abstract setting can yield extremely insightful results regarding their spectral properties.
The understanding of such properties is integral to the study of numerical algorithms as an entire family, rather than on a case-by-case basis (cf., e.g., \cite{sheng1993}).
While this task can be difficult in and of itself, the situation is further complicated by the fact that one is also often interested in estimating such families in a non-Hilbert space setting. 
In this case, it is not so obvious how one can immediately obtain analogous results due to the lack of structure exhibited by many such spaces.

Adding to the aforementioned difficulties, there has been a recent trend in considering operators which exhibit so-called \textit{non-local features} (see, for instance,
\cite{padgett2020analysis,padgett2020anomalous,%
Gale2013,%
doi:10.1080/03605302.2017.1363229,doi:10.1080/03605300600987306} and references therein).
The Cauchy problems associated to these operators no longer exhibit solutions which are semigroups, which complicates any ensuing analysis. 
Therefore, even in the classical Hilbert space setting or settings with bounded operators, one cannot employ classical techniques to obtain an understanding of the spectral or norm properties of the solution operator.
This difficulty is a primary motivation of the current study.

To this end, herein, we employ techniques which allow for the development of results which mirror the spectral results often derived in Hilbert spaces for non-local problems in arbitrary normed vector spaces.
While such problems have numerous open questions associated to them, we currently focus on the task of developing norm estimates of the solution operators (which will contain classical strongly continuous semigroups as a special case).
This goal is accomplished by introducing a particular semi-inner-product on a given normed vector space. 
This semi-inner-product differs from the classical ones originally studied by Lumer (cf.\ \cite{lumer1961semi,giles1967classes,%
dragomir2004semi}), but exhibit desirable properties which allow for the development of the desired norm bounds (cf.\ \cref{lem:cont,lem:nice_props}).
A nice consequence of these newly derived results is briefly outlined in \cref{cor:main}.
An important fact worth mentioning is that \cref{lem:nice_props,cor:main} demonstrate that while the non-local problems have solution operators which lack certain desired features, it is the case that their norm estimates may be viewed as continuous perturbations of the classical setting with the perturbations being comparable to the amount of non-locality present (e.g., the value of $\alpha \in (0,1]$ in \cref{cor:main}).

The remainder of this article is organized as follows. In \cref{sec:2} we introduce a particular notion of a semi-inner-product on arbitrary normed vector spaces and introduce its associated so-called logarithmic norm. In \cref{sec:2_3} we use these ideas to prove norm growth bounds on strongly continuous semigroups (cf.\ \cref{def:c0}). In \cref{sec:4} we introduce a notion of fractional semigroups, whose construction depend upon the classical Mittag-Leffler functions (cf.\ \cref{def:mittag,def:sol1}). Thereafter, these notions are then used to prove analogous growth bounds for strongly continuous fractional semigroups in \cref{sec:4_2}. Finally, in \cref{sec:final} we provide some concluding remarks on the subject and outline potential future research directions.


\section{Background}\label{sec:2}

For the purposes of the ensuing analysis, a particular semi-inner-product is defined on arbitrary normed vector spaces. Through it, an
associated logarithmic norm is introduced. The necessary concepts of strongly continuous semigroups and their 
generators will be introduced and studied. An interesting growth bound result of such semigroups employing the logarithmic norm
will be proven (cf.\ \cref{lem:cont}). Our result appears in various forms in the known literature, although, to our knowledge, no formal proof has been provided which is independent of higher regularity in infinite-dimensional spaces (cf., e.g., \cite{soderlind2006logarithmic,dahlquist1958stability,%
lozinskii1958error,10.2307/2156188}).

Throughout this article, let $\C$ and $\R$ be the usual complex and real number fields, respectively, and let $i = \sqrt{-1}.$ Furthermore,
let $\N_0 = \{0,1,2,\ldots\}.$
In addition, we briefly mention a particular notation used throughout this article which emphasizes how various outside results are applied. If, for example, we cite a result which names a mathematical object $\mathcal{X}$, in order to state results about a family of objects, e.g., $\mathcal{Y}_t,$ $t\in\R,$ we will write ``applied for every $t\in\R$ with $\mathcal{X} \with \mathcal{Y}_t$ in the notation of \ldots'' in order to clarify its use.

\begin{setting}\label{setting}
Let $X$ be a vector space defined over the field $\R$ endowed with a proper norm $\norm{\cdot}.$ We denote such a normed $\R$-Banach space 
as $(X,\norm{\cdot}).$ We also define the following.
\begin{enumerate}[label=(\roman*)]
\item Let $B(X)=\{A\colon X \to X \colon \norm{Ax} <\infty, \ \forall\, x\in X;\ \norm{x} =1\}.$
\item Let $\id \colon X \to X$ satisfy that $\id x = x,\ \forall\, x \in X.$ 
\item Let $\domain(A) \subseteq X$ be the domain of $A \colon X \to X.$ 
\item Let $\resolve(A) = \{A\colon  X\rightarrow X \colon \exists\, (\lambda \id - A)^{-1} \in B(X),\ \lambda\in\C \}$ be the
resolvent set of $A.$
\item Let $\real(z) = \frac{1}{2}\pr[]{ z + \overline{z}},\ \forall z\in\C.$
\item  Let $\Gamma(z) = \int_0^\infty x^{z-1} \exp(-x) \, dx,\ z\in \C\cup\R \backslash \{ \ldots , -2, -1, 0 \},$ be the
Gamma function.
\item Let $R_A \colon \C \to B(X)$ satisfy that $R_A(\lambda) = (\lambda \id - A)^{-1}$, $\forall\, A\colon X \to X$, $\forall\, \lambda\in\resolve(A)$.
\end{enumerate}
Finally, for every $v \colon [0,\infty)\rightarrow X$ we define
$D_t^+ v(t) = \limsup_{\varepsilon\to 0^+} \frac{v(t+\varepsilon) - v(t)}{\varepsilon} \in [-\infty,\infty], \ t\in[0,\infty).$
\end{setting}


\subsection{Logarithmic norms on Banach spaces}\label{sec:2_2}

In \cref{def:semi} we introduce a particular notion of a semi-inner-product on normed vector spaces. Semi-inner-products have long been studied in classical analysis (cf., e.g., \cite{lumer1961semi,giles1967classes,%
dragomir2004semi}) and were originally introduced by Lumer in an attempt to extend classical Hilbert space arguments to more general spaces.
While we will not explore all possible properties of semi-inner-products, it is worthwhile to mention that the primary difference between classical inner products and semi-inner-products is the fact that the latter is not necessarily uniquely defined nor bilinear.

\begin{definition}
[Right defined semi-inner-product]
\label{def:semi} 
Let $(X,\norm{\cdot})$ be a normed $\R$-Banach space.
We denote by $\inner{\cdot,\cdot} \colon X \times X \to \R$ the function which satisfies for all $v,w\in X$ that
\begin{equation}\label{eq:limit}
\inner{v,w} = \left[ \lim_{\varepsilon\to 0^+} \frac{\norm{w + \varepsilon v } - \norm{w}}{\varepsilon} \right] \norm{w}.
\end{equation}
\end{definition}

We note that the limit in \cref{eq:limit} exists due to the fact that the underlying norm on $X$ possesses one-sided Gateaux differentials (cf., e.g., \cite{abatzoglou1979norm}). We also wish to note that the choice of the limit employed in \cref{eq:limit} is somewhat arbitrary; in fact, there are numerous other choices which would have demonstrated desirable properties, but the inclusion of these other choices would not have enriched the following discussions.

\begin{lemma}
\label{lem:semi_inner_props}
Let $(X,\norm{\cdot})$ be a normed $\R$-Banach space. Then
\begin{enumerate}[label=(\roman*)]
\item\label{lem:semi_i1} it holds for all $u,v \in X$ that $\inner{u,v} \le \norm{u}\norm{v},$
\item\label{lem:semi_i2} it holds for all $u \in X$ that $\norm{u}^2 = \inner{u,u},$
\item\label{lem:semi_i4} it holds for all $c \in [0,\infty),$ $u,v \in X$ that $\inner{cu, v} = c \inner{u,v},$
and
\item\label{lem:semi_i5} it holds for all $u,v,w \in X$ that $\inner{u+v,w} \le \inner{u,w} + \inner{v,w}$
\end{enumerate}
(cf.\ \cref{def:semi}).
\end{lemma}

\begin{proof}
[Proof of \cref{lem:semi_inner_props}]
First, note that \cref{def:semi} and the triangle inequality ensure that for all $u,v\in X$ it holds that
\begin{equation}
\inner{u,v} = \br[\bigg]{ \lim_{\varepsilon\to 0^+} \frac{\norm{v + \varepsilon u} - \norm{v}}{\varepsilon} } \norm{v} \le \norm{u} \norm{v}.
\end{equation}
This 
establishes \cref{lem:semi_i1}. Next 
observe that \cref{lem:semi_i2} immediately follows from \cref{def:semi}.
In addition, note that \cref{def:semi} assures that for all $c\in (0,\infty),$ $u,v \in X$ it holds that
\begin{equation}
\begin{split}
\inner{ c u,v} & = \br[\bigg]{ \lim_{\varepsilon\to 0^+} \frac{\norm{v + \varepsilon (cu)} - \norm{v}}{\varepsilon} } \norm{v} \\
& = c \br[\bigg]{ \lim_{c\varepsilon\to 0^+} \frac{\norm{v + (c\varepsilon) u} - \norm{v}}{c \varepsilon} } \norm{v} = c \inner{u,v}.
\end{split}
\end{equation}
Combining this with the fact that \cref{def:semi} implies that for all $u,v\in X$ it holds that $\inner{0u,v} = \inner{0,v} = 0$ establishes \cref{lem:semi_i4}.
Moreover, note that \cref{def:semi} and the triangle inequality
demonstrate that for all $u,v,w\in X$ it holds that
\begin{align}
& \inner{u+v,w} = \br[\bigg]{ \lim_{\varepsilon\to 0^+} \frac{\norm{w + \varepsilon (u+v)} - \norm{w}}{\varepsilon} } \norm{w} \nonumber \\
& \le \br[\bigg]{ \lim_{\varepsilon\to 0^+} \frac{\norm{w + 2\varepsilon u} - \norm{w} }{2\varepsilon} } \norm{w} + \br[\bigg]{ \lim_{\varepsilon\to 0^+} \frac{\norm{w + 2\varepsilon v} - \norm{w} }{2\varepsilon}} \norm{w} \\
& 
= \inner{u,w} + \inner{v,w} .\nonumber
\end{align}
This establishes \cref{lem:semi_i5}.
The proof of \cref{lem:semi_inner_props} is thus completed.
\end{proof}

\begin{lemma}\label{lem:log_der}
Assume \cref{setting}. Then for all differentiable $v \colon [0,\infty) \to X,$ $t\in[0,\infty)$ it holds that
$
\norm{v(t)}^{-2} D_t^+ \norm{v(t)} =\inner[]{\frac{d}{dt}v(t), v(t)} \norm{v(t)}
$
(cf.\ \cref{def:semi}).
\end{lemma}

\begin{proof}
[Proof of \cref{lem:log_der}]
Throughout this proof let $v \colon [0,\infty) \to X,$ let $t\in[0,\infty),$ and assume without loss of generality that $\norm{v(t)} \neq 0.$
Observe that the hypothesis that $v$ is differentiable and Taylor's theorem (cf., e.g., Cartan et al.\ \cite[Theorem 5.6.3]{cartan2017differential}) yield that there exists $\delta_t(\varepsilon) \in X,$ $\varepsilon \in \R,$ such that for all $\varepsilon \in \R$ sufficiently small it holds that
\begin{enumerate}[label=(\Alph*)]
\item\label{aa} it holds that $v(t+\varepsilon) = v(t) + \varepsilon \frac{d}{dt} v(t) + \abs{\varepsilon} \delta_t(\varepsilon)$ and
\item\label{bb} it holds that $\lim_{\varepsilon\to 0} \delta_t(\varepsilon) = 0$.
\end{enumerate}
Combining \cref{aa,bb} with \cref{def:semi} hence shows that
\begin{align}
& D_t^+ \norm{v(t)} = \limsup_{\varepsilon \to 0^+} \frac{\norm{v(t) + \varepsilon \frac{d}{dt} v(t) + \abs{\varepsilon} \delta_t(\varepsilon)} - \norm{v(t)}}{\varepsilon} \nonumber \\
& \quad = \lim_{\varepsilon \to 0^+} \frac{\norm{v(t) + \varepsilon \frac{d}{dt} v(t)} - \norm{v(t)}}{\varepsilon} \\
& \quad = \br[\bigg]{ \lim_{\varepsilon \to 0^+} \frac{\norm{v(t) + \varepsilon \frac{d}{dt} v(t)} - \norm{v(t)}}{\varepsilon} } \frac{\norm{v(t)}^2}{\norm{v(t)}^2}
= \frac{ \inner[\big]{\frac{d}{dt}v(t), v(t)} }{\norm{v(t)}^2} \norm{v(t)}. \nonumber
\end{align}
The proof of \cref{lem:log_der} is thus completed.
\end{proof}

\begin{definition}
[Logarithmic norm]
\label{def:log}
Assume \cref{setting}.
Then we denote
by $\logNorm \colon X \to \R$ the function which satisfies for all $A \colon X \to X,$ $v \in \domain(A) \subseteq X$ that
\begin{equation}
\logNorm(A) = \sup_{v \in \domain(A) \subseteq X} \frac{\inner{Av,v}}{\norm{v}^2}.
\end{equation}
\end{definition}

We close \cref{sec:2_2} with a brief discussion of \cref{def:log}. Let $N \in \N = \{1,2,3,\allowbreak\ldots\}$, $p \in [1,\infty)$, let $X = L_p([0,\infty);\R^{N})$ be the standard $L_p$-space on $\R^N$ (endowed with its typical norm), and let $A \colon \R^N \to \R^N$. Then \cref{def:semi,def:log} ensure that
\begin{equation}\label{log_expand}
\begin{split}
\sup_{v \in \domain(A) \subseteq X} \frac{\inner{Av,v}}{\norm{v}^2} & = \sup_{v \in \domain(A) \subseteq X} \frac{ \left[ \lim_{\varepsilon\to 0^+} \frac{\norm{v + \varepsilon Av } - \norm{v}}{\varepsilon} \right] \norm{v}}{\norm{v}^2} \\
& =  \lim_{\varepsilon\to 0^+} \frac{\lVert \id + \varepsilon A \rVert_{\text{op}} - 1}{\varepsilon},
\end{split}
\end{equation}
where $\lVert\cdot\rVert_{\text{op}}$ is the matrix norm induced by $\norm{\cdot}$,
which corresponds to the classical finite-dimensional definition used in, e.g., \cite{10.2307/2156188}.
The usefulness of such formulations can be seen from the resulting closed-form expressions for \cref{log_expand} for particular choices of $p \in [1,\infty)$. For instance, when $p=2$, it follows from direct calculation that $\logNorm(A) = \max\{ \lambda \in \C \colon \exists\, x \in \R^N \text{ such that } (A+A^T)x = 2 \lambda x \}$.

\subsection{Logarithmic norm bounds of classical semigroups}\label{sec:2_3}

We now briefly introduce the classical concepts of strongly continuous semigroups and their associated generators (cf.\ \cref{def:c0,def:gen}).
For a more detailed background and applications of these objects, we refer readers to, e.g., \cite{sheng1993,sinha2017theory}.
The main result of \cref{sec:2_3} is \cref{lem:cont}, which provides a means to estimate norm bounds for strongly continuous semigroups in a useful manner. In particular, we have that for every generator of a strongly continuous semigroup, $A \colon X \to X$, \cref{lem:cont} implies that if $\logNorm(A) \in (-\infty,0]$ it holds that $A$ generates a so-called \textit{contraction semigroup} (cf., e.g., \cite[Page 10]{robinson1982basic}).

\begin{definition}
[Strongly continuous semigroup]
\label{def:c0}
Assume \cref{setting}.
Then $\semigroup \colon [0,\infty) \allowbreak \to B(X)$ is a strongly continuous semigroup if
\begin{enumerate}[label=(\roman*)]
\item it holds that $\semigroup_0 = \id,$
\item\label{def:c0_i2} it holds for all $s,t \in [0,\infty)$ that $\semigroup_{t+s} = \semigroup_t\semigroup_s,$ and
\item it holds for all $x \in X$ that $\lim_{t\to 0^+} \norm{\semigroup_t x - x } = 0$.
\end{enumerate}
\end{definition}

\begin{definition}
[Infinitesimal generator]
\label{def:gen}
Assume \cref{setting}
and let $\semigroup \colon [0,\infty) \to B(X)$ be a strongly continuous semigroup (cf.\ \cref{def:c0}). Then $A \colon \domain(A) \subseteq X \to X$ is the infinitesimal generator of $\semigroup \colon [0,\infty) \to B(X)$ if
\begin{enumerate}[label=(\roman*)]
\item it holds that $\domain(A) = \cu{ x \in X \colon \exists\, \lim_{h \to 0^+} h^{-1} \pr[]{ \semigroup_h - \id } x }$ and
\item it holds for all $x \in \domain(A)$ that $Ax = \lim_{h\to 0^+} h^{-1} \pr[]{ \semigroup_h - \id } x $.
\end{enumerate}
\end{definition}

\begin{lemma}\label{lem:cont}
Assume \cref{setting}
and let $A \colon \domain(A) \subseteq X \to X$ be the generator of a strongly continuous semigroup $\semigroup_t(A) \in B(X)$, $t\in[0,\infty)$ (cf.\ \cref{def:gen,def:c0}).
Then it holds for all $t\in[0,\infty),$ $x\in \domain(A)$ that
$\norm{\semigroup_t(A) x} \le \exp(t\mu_X(A)) \norm{x}$
(cf.\ \cref{def:log}).
\end{lemma}

\begin{proof}[Proof of \cref{lem:cont}]
First, observe that the fact that $A$ generates a strongly continuous semigroup and, e.g., Fetahu \cite[Theorem 2.12, item c)]{fetahu2014semigroups} (applied for every $t\in[0,\infty)$ with $A \with A$, $x \with x$, $T(t) \with \semigroup_t(A)$ in the notation of \cite[Theorem 2.12]{fetahu2014semigroups})
assure that 
\begin{enumerate}[label=(\Roman*)]
\item\label{A} for all 
$t\in[0,\infty),$
$x\in \domain(A)$ it holds that $ \semigroup_t(A)x \in \domain(A)$ and
\item\label{B} for all $t\in[0,\infty),$ $x\in \domain(A)$ it holds that $\frac{d}{dt}\semigroup_t(A)x = A \semigroup_t(A)x .$
\end{enumerate}
Next note that \cref{A,B} and \cref{lem:log_der}
hence demonstrate that for all $t\in[0,\infty),$ $x\in \domain(A)$ it holds that
\begin{equation}
\begin{split}
D_t^+ \norm{\semigroup_t(A)x} 
& = \frac{ \inner[\big]{ \frac{d}{dt} \semigroup_t(A)x, \semigroup_t(A)x} }{ \norm{\semigroup_t(A) x}^2} \norm{\semigroup_t(A) x} \\
& = \frac{ \inner[\big]{ A \semigroup_t(A)x, \semigroup_t(A)x} }{ \norm{\semigroup_t(A) x}^2} \norm{\semigroup_t(A) x}.
\end{split}
\end{equation}
This, \cref{A}, the fact that for all $x \in \domain(A)$ it holds that $\lim_{t\to 0^+} \norm{ \semigroup_t(A) x } = \norm{x},$ and, e.g., Szarski \cite[Theorem 9.6]{szarski1965differential} (applied for every $t\in[0,\infty),$ $x\in\domain(A)$ with $y(t) \with \norm{\semigroup_t(A)x},$ $\sigma(t,y) \with \inner{A\semigroup_t(A)x,\semigroup_t(A)x}\norm{\semigroup_t(A)x}^{-2}\norm{\semigroup_t(A)x}$ in the notation of \cite[Theorem 9.6]{szarski1965differential}) prove that for all $t\in[0,\infty),$ $x\in \domain(A)$ it holds that
\begin{equation}
\norm{ \semigroup_t(A) x } \le \exp\pr[\bigg]{ \frac{t \inner[\big]{ A \semigroup_t(A)x, \semigroup_t(A)x} }{ \norm{\semigroup_t(A) x}^2} } \norm{x} .
\end{equation}
This, 
the fact that $\R \ni x \to \exp(x) \in (0,\infty)$ is increasing,
and \cref{def:log} show that for all $t\in[0,\infty),$ $x \in \domain(A)$ it holds that
\begin{align}
\norm{ \semigroup_t(A) x } & \le \exp\pr[\bigg]{ \frac{t \inner[\big]{ A \semigroup_t(A)x, \semigroup_t(A)x} }{ \norm{\semigroup_t(A) x}^2} } \norm{x} 
\le
\sup_{y \in \domain(A)} \br[\Bigg]{ \exp\pr[\bigg]{ \frac{t\inner{Ay,y}}{\norm{y}^2} } \norm{x} } \nonumber \\
& \le
\exp\pr[\bigg]{ \sup_{y \in \domain(A)} \frac{t\inner{Ay,y}}{\norm{y}^2} } \norm{x} = \exp(t\logNorm(A)) \norm{x}.
\end{align}
The proof of \cref{lem:cont} is thus completed.
\end{proof}


\section{Fractional semigroups}\label{sec:4}

We now introduce the novel solution operators of interest.
We note that the use of the term {\em fractional semigroup} is a bit of a misnomer as the objects in \cref{def:sol1} do not satisfy \cref{def:c0_i2} of \cref{def:c0} (cf., e.g., \cite[Theorem 3.3]{guswanto2015properties}). However, we use this term, herein, due to the fact that the classical strongly continuous semigroup is contained as a special case of the operators proposed in \cref{def:sol1}. 
In \cref{sec:4_1} we introduce the concept of the two-parameter Mittag-Leffler and Wright functions. These functions allow for the convenient representation of the fractional semigroups in a fashion which is analogous to the classical functional calculus methods used to represent strongly continuous semigroups.
The main result of this article is \cref{lem:nice_props} of \cref{sec:4_2}. This result can be seen as a generalization of \cref{lem:cont} and, as such, will have analogous useful implications in the study of singular integral problems.

\subsection{Mittag-Leffler and Wright functions}\label{sec:4_1}

In \cref{def:mittag,def:wright} we introduce two families of functions which may be viewed as generalizations of the exponential function and the Bessel functions, respectively. 
There exists a rich theory behind each of the special functions, but such explorations are tangential to our current goal. As such, we outline their properties which are germane to the current study in \cref{lem:mittag_props,lem:wright}.

\begin{definition}
[Mittag-Leffler function]
\label{def:mittag}
Assume \cref{setting} and
let $\alpha, \beta \in \C$ satisfy that $\real(\alpha)\in (0,\infty).$ Then we denote
by $\mittag_{\alpha,\beta} \colon \C \to \C$ the function which satisfies for all $z\in\C$ that
\begin{equation}
\mittag_{\alpha,\beta}(z) = \sum_{k=0}^\infty \frac{z^k}{\Gamma(\beta + \alpha k)}.
\end{equation}
\end{definition}

\begin{lemma}\label{lem:mittag_props}
Let $\alpha \in (0,1],$ $\lambda \in \C.$ Then
\begin{enumerate}[label=(\roman*)]
\item\label{lem:mittag_props_i1} it holds for all $z\in \C$ that $\mittag_{1,1}(z) = \exp(z)$ and
\item\label{lem:mittag_props_i2} it holds for all $z\in (0,\infty)$ that $\frac{d}{dz} \mittag_{\alpha, 1}(\lambda z^\alpha) = \lambda z^{\alpha-1} \mittag_{\alpha,\alpha}(\lambda z^\alpha)$
\end{enumerate}
(cf.\ \cref{def:mittag}).
\end{lemma}

\begin{proof}
[Proof of \cref{lem:mittag_props}]
First, note that \cref{lem:mittag_props_i1} follows directly from \cref{def:mittag}. Next observe that \cref{def:mittag} and, e.g., Rudin \cite[Theorem 7.17]{rudin1964principles}
ensure that for all $z\in (0,\infty)$ it holds that
\begin{equation}
\frac{d}{dz} \mittag_{\alpha,1}(\lambda z^\alpha) = \frac{d}{dz} \sum_{k=0}^\infty \frac{(\lambda z^\alpha)^k}{\Gamma(1 + \alpha k)} = \sum_{k=0}^\infty \frac{d}{dz} \frac{\lambda^k z^{\alpha k}}{\Gamma(1 + \alpha k)} = \sum_{k=1}^\infty \frac{\alpha k \lambda^k z^{\alpha k - 1}}{\Gamma(1 + \alpha k)}.
\end{equation}
This, the fact that for all $z \in \C$ such that $\real(z) \in \R \backslash\{\ldots, -2,-1,0\}$ it holds that $z\Gamma(z) = \Gamma(z+1),$ and \cref{def:mittag} assure that for all $z\in (0,\infty)$ it holds that
\begin{align}
\frac{d}{dz} \mittag_{\alpha,1}(\lambda z^\alpha) & = \sum_{k=1}^\infty \frac{ \lambda^k z^{\alpha k - 1}}{\Gamma(\alpha k)} = \sum_{k=0}^\infty \frac{ \lambda^{k+1} z^{\alpha (k+1) - 1}}{\Gamma(\alpha (k+1))} \\
& = \lambda z^{\alpha-1} \sum_{k=0}^\infty \frac{ \lambda^k z^{\alpha k} }{ \Gamma( \alpha + \alpha k)} = \lambda z^{\alpha-1} \sum_{k=0}^\infty \frac{ ( \lambda z^{\alpha})^k }{ \Gamma( \alpha + \alpha k)} = \lambda z^{\alpha-1} \mittag_{\alpha,\alpha}( \lambda z^{\alpha} ). \nonumber 
\end{align}
This establishes \cref{lem:mittag_props_i2}.
The proof of \cref{lem:mittag_props} is thus completed.
\end{proof}

\begin{definition}
[Wright function]
\label{def:wright}
Assume \cref{setting}.
Then we denote by $\wright_{\alpha,\beta} \colon \C \to \C$ the function which satisfies for all $\alpha \in (-1,\infty),$ $\beta, z \in \C$ that
\begin{equation}
\wright_{\alpha,\beta}(z) = \sum_{k=0}^\infty \frac{(-z)^k}{\Gamma(k+1)\Gamma(\beta + \alpha k)}.
\end{equation}
\end{definition}

\begin{lemma}\label{lem:wright}
Let $\alpha \in (0,1).$ Then
\begin{enumerate}[label=(\roman*)]
\item\label{lem:wright_i3} it holds for all $\beta, z \in \C$ that $\mittag_{\alpha,\beta}(-z) = \int_0^\infty \wright_{-\alpha,\beta-\alpha}(t) \exp(-tz) \, dt,$
\item\label{lem:wright_i2} it holds for all $\beta\in\C,$ $n \in \N_0 $ that $\int_0^\infty z^n \wright_{-\alpha,\beta}(z) \, dz = \frac{\Gamma(1+n)}{\Gamma(\beta + \alpha(1+n))} ,$ and
\item\label{lem:wright_i1} it holds for all $\beta ,z \in [0,\infty)$ that $\wright_{-\alpha,\beta}(z) \in [0,\infty)$
\end{enumerate}
(cf.\ \cref{def:mittag,def:wright}).
\end{lemma}

\begin{proof}
[Proof of \cref{lem:wright}]
Throughout this proof for every $f \colon [0,\infty) \to \C$ let $\mathcal{L}[f;z] \in \C,$ $z\in \C,$ satisfy for all $z\in\C$ that 
$\mathcal{L}[f;z] = \int_0^\infty f(t) \exp(-tz) \, dt.$
Note that, e.g., Wright \cite{wright1940generalized} ensures that for all $\beta, z\in \C$ it holds that
\begin{equation}
\wright_{-\alpha,\beta-\alpha}(z) = \frac{1}{2\pi i} \int_{\text{Ha}} \exp(\zeta - z\zeta^{\alpha}) \zeta^{\alpha-\beta} \, d\zeta,
\end{equation}
where Ha denotes the Hankel path in the $\zeta$-plane with a cut along the negative real semi-axis $\operatorname{arg}\zeta = \pi$ (cf., e.g., \cite[Section 13.2.4]{krantz1999handbook}).
This and Fubini's theorem imply that for all $\beta, z\in\C$ it holds that
\begin{align}
& \int_0^\infty \wright_{-\alpha,\beta-\alpha}(t) \exp(-tz) \, dt = \int_0^\infty \br[\bigg]{ \frac{1}{2\pi i} \int_{\text{Ha}} \exp(\zeta - t\zeta^{\alpha}) \zeta^{\alpha-\beta} \, d\zeta } \exp(-tz) \, dt \nonumber \\
& = \frac{1}{2\pi i} \int_{\text{Ha}} \frac{\exp(\zeta)}{ \zeta^{\beta-\alpha} } \br[\bigg]{ \int_0^\infty \exp\pr[\big]{ -t ( z + \zeta^\alpha ) } \, dt } \, d\zeta 
= \frac{1}{2\pi i} \int_{\text{Ha}} \frac{ \exp(\zeta) \zeta^{\alpha-\beta} }{ z + \zeta^\alpha } \, d\zeta .
\end{align}
Combining this and the integral representation of the generalized Mittag-Leffler function as seen in, e.g., Gorenflo and Mainardi \cite[Eq. (A.21)]{gorenflo1997fractional} shows that for all $\beta, z\in\C$ it holds that
\begin{equation}
\int_0^\infty \wright_{-\alpha,\beta-\alpha}(t) \exp(-tz) \, dt = \mittag_{\alpha,\beta}( - z).
\end{equation}
This establishes \cref{lem:wright_i3}.
Moreover, observe that \cref{def:wright}, 
the fact that for all $\beta \in \C,$ $n \in \N_0,$ $f \colon [0,\infty) \to \C$ it holds that
$\int_0^\infty t^n f(t) \, dt = \lim_{z \to 0} \, (-1)^n \frac{d^n}{dz^n} \mathcal{L}[f;z]$
(cf., e.g., Korn and Korn \cite[Table 8.3-1]{korn2000mathematical}),
and \cref{lem:wright_i3}
guarantee that for all $\beta \in \C,$ $n\in\N_0$ it holds that
\begin{equation}\label{eq:4_10}
\int_0^\infty z^n \wright_{-\alpha,\beta}(z) \, dz = \lim_{t \to 0} \, (-1)^n \frac{d^n \mathcal{L}[\wright_{-\alpha,\beta};t]}{dz^n}  = \lim_{t \to 0} \, (-1)^n \frac{d^n \mittag_{\alpha,\alpha+\beta}(-t) }{dt^n} .
\end{equation}
Next note that induction on $n\in\N_0,$ the dominated convergence theorem, and \cref{def:mittag} show that for all $\beta\in\C,$ $n\in\N_0$ it holds that
\begin{align}
\frac{d^n}{dt^n} E_{\alpha,\alpha+\beta}(-t) & = \frac{d^n}{dt^n} \sum_{k=0}^\infty \frac{ (-t)^k }{ \Gamma( (\alpha+\beta) + \alpha k) } = \sum_{k=0}^\infty \frac{d^n}{dt^n} \frac{ (-t)^k }{ \Gamma( (\alpha+\beta) + \alpha k) } \nonumber \\
& = \sum_{k=n}^\infty \frac{ (-1)^k \br[\big]{ k \cdot (k-1) \cdot \ldots \cdot (k-n+1) } t^{k-n} }{ \Gamma( \beta + \alpha (1+k) ) } .
\end{align}
Combining this, 
e.g., Rudin \cite[Theorem 7.17]{rudin1964principles}, the definition of the Gamma function (cf.\ \cref{setting}), and \cref{eq:4_10} ensures that for all $\beta\in\C,$ $n\in\N_0$ it holds that
\begin{align}
& \int_0^\infty z^n \wright_{-\alpha,\beta}(z) \, dz = \lim_{t\to 0} \, (-1)^n \br[\Bigg]{ \sum_{k=n}^\infty \frac{ (-1)^k \br[\big]{ k \cdot (k-1) \cdot \ldots \cdot (k-n+1) } t^{k-n} }{ \Gamma( \beta + \alpha (1+k) ) } } \nonumber \\
& = \sum_{k=n}^\infty \br[\bigg]{ \lim_{t\to 0} \frac{ (-1)^{n+k} \br[\big]{ k \cdot (k-1) \cdot \ldots \cdot (k-n+1) } t^{k-n} }{ \Gamma( \beta + \alpha (1+k) ) } } 
= \frac{\Gamma(1 + n)}{ \Gamma( \beta + \alpha(1+n) ) }.
\end{align}
This establishes \cref{lem:wright_i2}.
In addition, observe that \cref{def:wright} 
and, e.g., \cite[Theorem 1.3-7]{djrbashian1993harmonic} (applied for every $\beta,z\in[0,\infty)$ with $\rho \with -\nicefrac{1}{\alpha},$ $\mu \with \beta-\alpha,$ $\Phi \with \wright$ in the notation of \cite[Theorem 1.3-6]{djrbashian1993harmonic}) guarantees that \cref{lem:wright_i1} holds.
The proof of \cref{lem:wright} is thus completed.
\end{proof}

\subsection{Logarithmic norm bounds of fractional semigroups}\label{sec:4_2}

We are now in a position to prove \cref{lem:nice_props}, which is the main result of the article.
Note that \cref{def:sol1} provides the definition of the two-parameter fractional semigroups which are the focus of the current study. It should be clear that these objects are simply the infinite-dimensional counterparts to the classical Mittag-Leffler functions introduced in \cref{def:mittag}. 

To the authors' knowledge, there have been no result along the lines of \cref{lem:nice_props} presented in the literature---neither in finite dimensions nor for particular cases of the operator $A.$
It is particularly interesting to note that when $\alpha \in (0,1)$ it holds that \cref{lem:nice_props} implies the fractional semigroup exhibits sub-exponential growth/decay (which differs from the exponential growth/decay of classical semigroups).
We close \cref{sec:4_2} with \cref{cor:main} which demonstrates the usefulness of these logarithmic norm bounds in a setting which occurs quite often in applied mathematics. 
Note that an implication of \cref{cor:main} is that if $A$ generates a contraction semigroup then it also exhibits a contraction fractional semigroup. 

\begin{definition}
[Fractional Semigroup]
\label{def:sol1}
Assume \cref{setting}.
Then for every $A \colon X \to X$ we denote by $\fractional_t^{\alpha,\beta}(A) \in B(X),$ $\alpha,\beta\in(0,1],$ $t\in[0,\infty),$ the function which satisfies for all $\alpha,\beta\in(0,1],$ $t\in[0,\infty)$ that
\begin{equation}
\fractional_t^{\alpha,\beta}(A) = \frac{1}{2\pi i} \int_{\mathcal{C}} \mittag_{\alpha,\beta}(\lambda t^\alpha) R_A(\lambda) \, d \lambda,
\end{equation}
where $\mathcal{C}\subseteq\C$ is any contour containing $\sigma(A) = \C \backslash \resolve(A)$ (cf.\ \cref{def:mittag}).
\end{definition}

\begin{theorem}\label{lem:nice_props}
Assume \cref{setting}
and let $A \colon \domain(A) \subseteq X \to X$ be the generator of a strongly continuous semigroup $\semigroup_t(A) \in B(X),$ $t\in[0,\infty)$
(cf.\ \cref{def:c0,def:gen}).
Then
it holds for all $\alpha,\beta \in (0,1],$ $t\in[0,\infty),$ $x\in \domain(A)$ that
\begin{equation}\label{eq:5}
\norm{\fractional_t^{\alpha,\beta}(A) x} \le \mittag_{\alpha,\beta}\pr[]{ t^\alpha \logNorm(A) } \norm{x}
\end{equation}
(cf.\ \cref{def:mittag,def:log,def:sol1}).
\end{theorem}

\begin{proof}
[Proof of \cref{lem:nice_props}]
Throughout this proof let $\mathcal{C} \subseteq \C$ be a contour containing $\sigma(A) = \C \backslash \rho_X(A).$
Note that the Riesz-Dunford functional calculus (cf., e.g., \cite[Theorem 13.5]{rudin2006real}), \cref{def:sol1}, and \cref{lem:mittag_props_i1} of \cref{lem:mittag_props} guarantee that for all $t\in[0,\infty),$ $x\in\domain(A)$ it holds that
\begin{equation}\label{eq:riesz}
\begin{split}
\semigroup_t(A)x & = \frac{1}{2\pi i} \int_{\mathcal{C}} \exp( \lambda t) R_A(\lambda) x \, d\lambda \\
& = \frac{1}{2\pi i} \int_{\mathcal{C}} \mittag_{1,1}( \lambda t) R_A(\lambda) x \, d\lambda = \fractional_t^{1,1}(A) x.
\end{split}
\end{equation}
This, \cref{def:mittag}, \cref{def:sol1}, \cref{lem:wright_i3} of \cref{lem:wright}, 
\cref{lem:mittag_props_i1} of \cref{lem:mittag_props},
and Fubini's theorem ensure that for all $\alpha\in(0,1),$ $\beta\in (0,1],$ $t\in[0,\infty),$ $x \in \domain(A)$ it holds that
\begin{align}
\fractional_t^{\alpha,\beta}(A) x & = \frac{1}{2\pi i} \int_{\mathcal{C}} \mittag_{\alpha,\beta}(\lambda t^\alpha) R_A(\lambda) x \, d\lambda \nonumber \\
&
= \frac{1}{2\pi i} \int_{\mathcal{C}} \br[\bigg]{ \int_0^\infty \Psi_{-\alpha,\beta-\alpha}(z) \exp(\lambda t^\alpha z) \, dz } R_A(\lambda) x \, d\lambda \nonumber \\
& = \frac{1}{2\pi i} \int_0^\infty \Psi_{-\alpha,\beta-\alpha}(z) \int_{\mathcal{C}} \exp(\lambda t^\alpha z) R_A(\lambda) x \, d\lambda \, dz \\
& = \frac{1}{2\pi i} \int_0^\infty \Psi_{-\alpha,\beta-\alpha}(z) \int_{\mathcal{C}} \mittag_{1,1}(\lambda t^\alpha z) R_A(\lambda) x \, d\lambda \, dz \nonumber \\
& = \int_0^\infty \Psi_{-\alpha,\beta-\alpha}(z) \fractional_{t^\alpha z}^{1,1}(A) x \, dz = \int_0^\infty \Psi_{-\alpha,\beta-\alpha}(z) \semigroup_{t^\alpha z}(A) x \, dz. \nonumber
\end{align}
Combining this, Jensen's inequality, \cref{lem:cont}, and \cref{lem:wright_i1,lem:wright_i3} of \cref{lem:wright} assures that for all $\alpha\in(0,1),$ $\beta\in (0,1],$ $t\in[0,\infty),$ $x \in \domain(A)$ it holds that
\begin{align}\label{eq:3_14}
& \norm{ \fractional_t^{\alpha,\beta}(A) x } = \norm[\bigg]{ \int_0^\infty \Psi_{-\alpha,\beta-\alpha}(z) \semigroup_{t^\alpha z}(A) x \, dz } \nonumber \\
& \quad \le \int_0^\infty \Psi_{-\alpha,\beta-\alpha}(z) \norm{ \semigroup_{t^\alpha z}(A) x } \, dz \\
& 
\quad \le \int_0^\infty \Psi_{-\alpha,\beta-\alpha}(z) \exp\pr[\big]{ t^\alpha z \logNorm(A) } \norm{x} \, dz 
= \mittag_{\alpha,\beta}\pr[\big]{ t^\alpha \logNorm(A) } \norm{x}. \nonumber 
\end{align}
Combining this, 
\cref{eq:riesz}, and \cref{lem:cont} establishes \cref{eq:5}.
The proof of \cref{lem:nice_props} is thus completed.
\end{proof}

An immediate consequence of \cref{lem:nice_props} is outlined in \cref{cor:main}.
While the proof of \cref{cor:main} is interesting, it is also brief and straightforward; hence, we omit the proof for brevity. 
Interested readers may note that this result follows by combining \cref{lem:nice_props} with the fact that the Dirichlet Laplace operator on $L^2([0,1];\R)$, denoted $A,$ is negative definite and the generator of a contraction semigroup and the fact that (in the notation of \cref{cor:main}) it holds for all $t\in[0,\infty),$ $x\in[0,1]$ that $u(t,x) = \fractional_t^{\alpha,1}(A)u(0,x)$ (cf., e.g., \cite[Lemma 3.1]{padgett2018quenching}).

\begin{corollary}\label{cor:main}
Let $\alpha \in (0,1],$ let $X = L^2([0,1];\R)$ be the $\R$-Hilbert space of of equivalence classes of Lebesgue square-integrable functions from $(0,1)$ to $\R$ equipped with its standard norm $\norm{\cdot},$
let $A \colon D(A) \subseteq X \to X$ be the Dirichlet Laplace operator,
for every $t\in[0,\infty),$ $x\in[0,1]$ let $u(t,x) \in D(A) \subseteq X$ satisfy that
$u(t,x) = u(0,x) + \frac{1}{\Gamma(\alpha)} \int_0^t (t-s)^{\alpha -1} A u(s,x) \, ds$
(cf.\ \cref{setting}).
Then
it holds for all $t\in[0,\infty),$ $x\in[0,1]$ that $\norm{ u(t,x) } \le \norm{u(0,x)}$
(cf.\ \cref{def:log,def:c0,def:gen}).
\end{corollary}

%


\section{Conclusions and future endeavors}\label{sec:final}

In this article we have explored norm bounds of fractional semigroups in arbitrary normed vector spaces.
To aid in this study, we introduced a particular semi-inner-product and used its properties to provide norm bounds of strongly continuous semigroups via the logarithmic norm. 
We then combined these results with representation results for fractional semigroups in order to obtain the main result (i.e., \cref{lem:nice_props}) of the article. This result is the first of its kind and demonstrates that the newly studied fractional semigroups exhibit growth/decay properties which are analogous to the classical setting---the difference being that the exponential function is replaced with a two-parameter Mittag-Leffler function (cf.\ \cref{def:mittag}). An immediate consequence of this result is that the contractive property of a semigroup is preserved in the fractional setting---a surprising result which will be quite useful in actual applications (e.g., the setting outlined in \cref{cor:main}).

Future endeavors in this direction will involve both continued abstract analysis and the implementation of actual numerical algorithms.
Of particular interest to the authors is the development of operator splitting algorithms for fractional semigroups. 
Such techniques are local by definition, hence the nonlocality of these new operators (i.e., their lack of a semigroup property) complicates matters considerably. 
However, one can demonstrate that there exists a semigroup-like property which the fractional semigroups satisfy.
From this, we will construct a generalization of the classical Lie-Trotter splitting method (cf., e.g., \cite{sheng1993}) and use the results developed herein to complete the study. Numerical endeavors will pursue the use of the presented operator theoretical representation of singular integral equations for actual simulations. Such methods have yet to be explored and initial results seem to indicate that they will provide a significant improvement in efficiency and accuracy.


\section*{Acknowledgement}
The second author acknowledges funding by the National Science Foundation (NSF 1903450). 
The third author would like to thank College of Arts and Sciences, Baylor University for partial support 
through a research leave award.
The first and second authors also gratefully acknowledge the impact that Lance Littlejohn has had on this project---Lance provided the their initial exposure to the beauty of operator theory during Lance's graduate functional analysis course at Baylor University.
All three authors would like to congratulate Lance Littlejohn on the occasion of his 70th birthday.


\bibliographystyle{acm}
\bibliography{bibfile}

\end{document}